\documentclass[12pt]{amsart}
\usepackage{latexsym,amsmath,amssymb,amsfonts,amscd,graphics}
\usepackage{blindtext}
\usepackage[left=1in,right=1in]{geometry}
\usepackage[skip=4pt plus1pt, indent=10pt]{parskip}
\usepackage[mathscr]{eucal}
\usepackage{mathrsfs}
\usepackage{tikz-cd}
\usepackage{hyperref, cleveref}
\usepackage{thmtools}
\usepackage[all]{xy}
\xyoption{rotate}
\hypersetup{
	colorlinks=true,
	linkcolor=blue,
	filecolor=blue,      
	urlcolor=cyan,
	citecolor=magenta}

\numberwithin{equation}{section}
\theoremstyle{plain}
\newtheorem{theorem}[equation]{Theorem}

\newtheorem{lemma}[equation]{Lemma}
\newtheorem{proposition}[equation]{Proposition}

\theoremstyle{remark}
\newtheorem{remark}[equation]{Remark}

\theoremstyle{definition}
\newtheorem{definition}[equation]{Definition}

\newcommand{\bP}{\mathbb{P}}
\newcommand{\bA}{\mathbb{A}}

\newcommand{\bZ}{\mathbb{Z}}

\newcommand{\calA}{\mathcal{A}}

\newcommand{\calF}{\mathcal{F}}

\newcommand{\calO}{\mathcal{O}}

\newcommand{\calI}{\mathcal{I}}

\newcommand{\calD}{\mathcal{D}}

\newcommand{\Ext}{\mathrm{Ext}}
\newcommand{\Aut}{\mathrm{Aut}}

\newcommand{\Hom}{\mathrm{Hom}}

\newcommand{\git}{/\kern-0.2em/}

\newcommand{\Ku}{\mathrm{Ku}}
\newcommand{\tX}{\widetilde{X}}
\newcommand{\Fix}{\mathrm{Fix}}

\DeclareMathOperator{\Bl}{Bl}
\DeclareMathOperator{\HH}{HH}

\title[Equivariant Kuznetsov Components]{Equivariant Kuznetsov components for Cubic fourfolds with a symplectic involution}

\author{Laure Flapan}
\address{Department of Mathematics\\
Michigan State University\\
619 Red Cedar Road, East Lansing, MI 48824}
\email{flapanla@msu.edu}

\author{Sarah Frei}
\address{Department of Mathematics, Dartmouth College, Kemeny Hall, Hanover, NH 03755}
\email{sarah.frei@dartmouth.edu}

\author{Lisa Marquand}
\address{Courant Institute,
  251 Mercer Street,
  New York, NY 10012, USA}
\email{lisa.marquand@nyu.edu}

\thanks{ 
L.F. was supported by NSF grant DMS-2200800.
S.F. was supported by NSF grant DMS-2401601.  
}



\begin{document}
\begin{abstract}
    We study the equivariant Kuznetsov component $\Ku_G(X)$ of a general cubic fourfold $X$ with a symplectic involution. We show that $\Ku_G(X)$ is equivalent to the derived category $D^b(S)$ of a $K3$ surface $S$, where $S$ is given as a component of the fixed locus of the induced symplectic action on the Fano variety of lines on $X$.
\end{abstract}
\maketitle
\section{Introduction}
The rationality problem of a cubic fourfold $X\subset \mathbb{P}^5$  is one of the most intensely studied problems in algebraic geometry.
A recent approach to this problem, pioneered by Kuznetsov \cite{Kuz10}, is to study rationality via the derived category $D^b(X)$.
More precisely, consider the Kuznetsov component $\Ku(X)$ given by the left orthogonal complement to the collection $\langle \calO_X, \calO_X(1), \calO_X(2)\rangle$ in $D^b(X)$. 
This component $\Ku(X)$ is a $K3$ category, 
meaning it has the same Hochschild homology as the derived category $D^b(S)$ of a $K3$ surface $S$ and 
its Serre functor is given by a shift by 2.
 Kuznetsov conjectured that the cubic fourfold $X$ is rational if and only if there is an equivalence $\Ku(X) \simeq D^b(S)$ for some $K3$ surface $S$.  This has been verified in all cases where  $X$ is known to be rational.

Cubic fourfolds with non-trivial automorphisms are a natural testing  ground for rationality conjectures.
An automorphism of a cubic fourfold $X$ is symplectic if it acts trivially on $H^{3,1}(X)$. If the group of symplectic automorphisms of $X$ has order greater than $2$, then there is a $K3$ surface $S$ such that $\Ku(X)\simeq D^b(S)$ \cite{Ouchi}. Moreover, if $X$ admits a symplectic automorphism of prime order $p\ge 3$, then $X$ is indeed rational \cite[Cor 1.3]{BGM25}. 
The focus of this paper is to study the Kuznetsov component in the remaining case --- the case of a cubic fourfold with a symplectic involution.

Such a cubic fourfold $X$ is potentially irrational in that $\Ku(X)\not\simeq D^b(S)$ for any $K3$ surface $S$ \cite[Theorem 1.2]{Marq23}. 
However, the equivariant Kuznetsov component $\Ku_G(X)$ of $G$-linearised objects is a 2-Calabi--Yau category \cite[Lemma 6.5]{beckmann2020}, so it is natural to ask whether $\Ku_G(X)$ is equivalent to the derived category of a $K3$ surface. 
Our main result is the following:

\begin{theorem}\label{main theorem}
    Let $X$ be a general cubic fourfold with a symplectic involution $\phi\in \Aut(X)$ and let $G:=\langle \phi\rangle \cong \bZ/2\bZ$. Then there is an equivalence of categories
    $$\Ku_G(X)\simeq D^b(S),$$
    where $S\subset F(X)$ is the $K3$ component of the fixed locus of the induced action of $G$ on the Fano variety of lines of $X$.
\end{theorem}

The category $\Ku_G(X)$ can be viewed as the Kuznetsov component of the smooth quotient $\tX/G,$ where $\tX$ is the blow up of $X$ in the fixed locus of $G$. Although \Cref{main theorem} does not address the rationality problem for the cubic fourfold $X$ itself, it sheds light on the question of when the quotient of a Fano variety is rational. Indeed, a key ingredient in the proof of \Cref{main theorem} is proving the rationality of the quotient $X/G$: 

\begin{proposition}[\Cref{prop: Z}]\label{prop:intro}
    The quotient $\tX/G$ is isomorphic to $\Bl_S(\bP^1\times \bP^3),$ where the $K3$ surface $S$ is embedded as a complete intersection of a $(3,0)$ divisor and a $(1,2)$ divisor. In particular, $X/G$ is rational.
\end{proposition}

There are few known examples of 2-Calabi--Yau categories beyond $K3$ and abelian surface categories. Thus, it is natural to ask whether $\Ku_G(X)$ gives new examples of 2-Calabi--Yau categories.  Although there are few instances in which $\Ku_G(X)$ has been computed, our result, together with results in \cite{hu2023} and \cite[\S 7.4]{Boa} for examples of cubic fourfolds with a $\mathbb{Z}/3\bZ$-action, gives evidence towards a negative answer to this question. 

In \cite{Boa}, Beckmann and Oberdieck study the equivariant derived category of a $K3$ surface $S$ with 
a finite group $G$ of symplectic automorphisms. 
They consider the action of $G$ on a moduli space $M$ of semi-stable objects (with respect to an invariant stability condition) on $S$ and show that, under certain additional hypotheses, there is a two-dimensional $G^\vee$-torsor $S'$ over the fixed locus ${\rm Fix}_G(M)$ such that $D^b_G(S)\simeq D^b(S')$ \cite[Theorem 1.1]{Boa}.
By \cite[Theorem 1.1]{LPZ}, the Fano variety of lines of a cubic fourfold $X$ is a moduli space of stable objects on the $K3$ category $\Ku(X)$. Thus, given that our case -- the case when $|G|=2$ -- is the only case when $X$ does not have an associated $K3$ surface $S$ such that ${\rm Ku}(X)\simeq D^b(S)$,  one can view \Cref{main theorem} as an extension of Beckmann--Oberdieck's result to the non-commutative setting. 

\begin{remark}
The $K3$ surface $S$ is equipped with an anti-symplectic involution. The induced action on $D^b(S)\cong \Ku_G(X)$ can be shown to commute with the residual action of the dual $G^\vee$ and hence lifts to an autoequivalence of $\Ku(X)$ \cite[Theorem 1.3]{Ela15} (see also \cite[Proposition 3.5]{beckmann2020}). This autoequivalence may be of independent interest, since it is of geometric origin but not induced by an automorphism of the cubic $X$.
\end{remark}

\subsection*{Method of proof}
We prove \Cref{main theorem} directly by constructing an equivalence of categories from $D^b(S)$ to $\Ku_G(X)$. 
We do this by exploiting the geometric situation: the $K3$ surface $S$ naturally parametrises lines that are invariant under the involution $\phi$. 
This allows us to produce a Fourier-Mukai kernel in $D^b(S\times X)$, defining a functor $\Phi$ from $D^b(S)$ to $\Ku(X)$. 
Equipping $S\times X$ with the diagonal $G$-action acting trivially on the first factor, and considering $D^b(S)$ as a $G$-category with trivial $G$-action, we prove that $\Phi$ is a $G$-functor, and hence factors through $\Ku_G(X).$ 
We use the criteria of Bondal and Orlov \cite[Proposition~7.1]{HuyFMBook} to prove the resulting functor $\Phi_G\colon D^b(S)\rightarrow \Ku_G(X)$ is fully faithful. In order to establish that $\Phi_G$ is an equivalence, it is then enough to show that $\Ku_G(X)$ is indecomposable. Since $K3$ categories are indecomposable, we in fact show that $\Ku_G(X)$ is a $K3$ category. 

Since $\Ku_G(X)$ appears as an admissible subcategory of the equivariant category $D^b_G(X)$ \cite[Theorem 6.3]{Ela11}, we prove that $\Ku_G(X)$ is a $K3$ category by studying $D^b_G(X)$. One classical approach to understanding $D^b_G(X)$ is via the Mckay correspondence \cite{BKR01}, which describes $D^b_G(X)$ by constructing a crepant resolution of the quotient $X/G$. This strategy was used recently in \cite{hu2023} for cubic fourfolds with a particular symplectic $\mathbb{Z}/3\bZ$-action and in \cite{casalainamartin2024} for cubic threefolds with a particular involution. However, this approach requires that the fixed locus of $G$ be equidimensional, which is not the case in our setting--- in our case the fixed locus of $G$ is the union of a surface and a line. 

Thus instead, we extend the action of $G$ to the blow up $\widetilde{X}:=\Bl_{\Fix_G(X)}X$. We then compute two semi-orthogonal decompositions of $D^b_G(\tX)$. The first comes from  using Orlov's blow up formula to compute $D^b(\tX)$ and taking its $G$-equivariant category. This semi-orthogonal decomposition realizes $\Ku_G(X)$ as an admissible subcategory of $D^b_G(\tX)$. The second semi-orthogonal decompositions of $D^b_G(\tX)$ comes from viewing the quotient stack $[\tX/G]$ as a square root stack and using the description of a semi-orthogonal decomposition of root stacks from \cite[Theorem 1.2]{BD24}. This allows us to conclude that $D^b(\tX/G)$ is an admissible subcategory of $D^b_G(\tX)$, which in turn by \Cref{prop:intro} realizes $D^b(S)$ as an admissible subcategory of $D^b_G(\tX)$. The fact that $\Ku_G(X)$ is a $K3$ category then follows by comparing the Hochschild homology of both decompositions.

\subsection*{Outline} In \Cref{sec: prelims}, we recall necessary preliminaries on derived categories and equivariant categories. In \Cref{sec:indecomp}, we prove that the equivariant Kuznetsov component for a cubic fourfold with a symplectic involution is a $K3$ category. Finally, in \Cref{sec:equivalence} we prove \Cref{main theorem} by explicitly constructing the necessary equivalence.

\subsection*{Acknowledgments}
This paper benefited from helpful correspondence and discussions with the
following people who we gratefully acknowledge: Nick Addington, Vanya Cheltsov, Enrico Fatighenti, Alex Perry, Saket Shah, Yu Shen, and Xiaolei Zhao.

\section{Preliminaries}\label{sec: prelims}
In this section we recall the necessary results on derived categories. We keep the exposition to the minimum needed for our purposes, providing references for the more general statements. In \Cref{subsec:CY} we introduce Calabi--Yau and $K3$ categories, and recall results on Hochschild homology. In \Cref{subsec: Gequiv} we discuss $G$-equivariant categories, specialising to when the group $G$ is induced by an automorphism of a variety $X$ acting on either $D^b(X)$ or on an admissible subcategory. In \Cref{subsec: root stack}, we recall the definition of a root stack and state a result on a semi-orthogonal decomposition for the derived category of a root stack.

\subsection{Calabi--Yau categories and Hochschild homology}\label{subsec:CY}

Let $\calD$ be a triangulated category over an algebraically closed field $k$ of characteristic 0. We say that $\calD$ is an \textbf{$n$-Calabi--Yau} category if both of the following hold: 
\begin{itemize}
    \item $\calD$ is an admissible subcategory of $D^b(X)$ for some quasi-projective variety $X$;
    \item the Serre functor is given by shift by $n$, i.e. $S_{\calD} = [n]$.
\end{itemize}

Let $\calD$ be either $D^b(X)$ for some smooth projective variety or an admissible subcategory of $D^b(X)$. We recall some results on Hochschild (co)homology of such a category -- the main reference is \cite{kuz09}, although the results in \cite{Macrisurvey} are sufficient for our purposes.

\begin{proposition}\cite[Prop 2.25]{Macrisurvey}\label{prop: HH additive}
    Let $\calD=\langle \calA_1, \dots \calA_r\rangle$ be a semi-orthogonal decomposition. Then for all $i\in \bZ$,
    $$\HH_i(\calD)\cong \bigoplus_{j=1}^r \HH_i(\calA_j).$$
\end{proposition}

\begin{lemma}\cite[Example 2.23]{Macrisurvey}\label{lem: HH of except}
    Let $E\in \calD$ be an exceptional object. Then $\HH_\bullet(\langle E\rangle)=k$ and is concentrated in degree 0.
\end{lemma}
For Calabi--Yau categories, the Hochschild homology and cohomology coincide, up to a shift: 
\begin{proposition}\label{prop: HH for CY}
    If $\calD$ is an $n$-Calabi--Yau category, then for all $i\in \bZ$ we have $$\HH^i(\calD)\cong \HH_{i-n}(\calD).$$
\end{proposition}

One can use Hochschild cohomology to define the notion of connectedness for a category $\calD$. A triangulated category $\calD$ that is an admissible subcategory of $D^b(X)$ for some quasiprojective variety $X$ is called \textbf{connected} if $\HH^0(\calD)=k.$

\begin{lemma}[Bridgeland's trick]\cite[Lemma 2.30]{Macrisurvey}\label{lem: Bridgelandtrick}
    Let $\calD$ be an n-Calabi--Yau variety. If $\calD$ is connected, then $\calD$ is indecomposable.
\end{lemma}
Finally, we can define $K3$ categories: a 2-Calabi--Yau category $\calD$ is called a \textbf{$K3$ category} if its Hochschild (co)homology coincides with the Hochschild (co)homology of a $K3$ surface. In particular, by \Cref{lem: Bridgelandtrick}, a $K3$ category is indecomposable.

A non-trivial example of a $K3$ category is the \textbf{Kuznetsov component} $\Ku(X)$ of a cubic fourfold $X\subset \bP^5$, defined as the right orthogonal complement of $\langle \mathcal{O}_X, \mathcal{O}_X(1), \mathcal{O}_X(2)\rangle$ in $D^b(X)$ \cite{Kuz10}. 

\subsection{$G$-equivariant categories}\label{subsec: Gequiv}
We recall some basic definitions of categorical actions and equivariant categories, following \cite{beckmann2020}. Throughout, let $\calD$ be a triangulated category, and $G$ a finite group.

\begin{definition}\cite[Definition 2.1]{beckmann2020}
    An {\bf action} $(\rho, \theta)$ of $G$ on $\calD$ consists of
    \begin{itemize}
        \item an auto-equivalence $\rho_g:\calD\rightarrow \calD$ for every $g\in G$; and 
        \item an isomorphism of functors $\theta_{g,h}:\rho_g\circ\rho_h\rightarrow \rho_{gh}$ for every pair $g,h\in G$,
    \end{itemize}
    such that the appropriate diagrams induced from the group law of $G$ commute.
\end{definition}
In this paper, $\calD$ will always be $D^b(X)$ for some variety $X$, or $\Ku(X)$ where $X$ is a cubic fourfold. The group $G$ will always be $\bZ/2\bZ$, and the action on $\calD$ will either be trivial or induced by an automorphism of $X$. Thus there is only one non-trivial auto-equivalence given by $\phi^*\colon \calD\rightarrow \calD$, where $\phi\in \Aut(X)$ is the generator of $G$.

\begin{definition}
    Let $G=\bZ/2\bZ$ with generator $\phi^*$  acting on $\calD$ as above. Then the {\bf equivariant category} $\calD_G$ is defined as follows:
    \begin{itemize}
        \item Objects are pairs $(E, \varphi)$ where $E\in \calD$ and $\varphi$ is an isomorphism (or linearisation)  $\varphi\colon E\xrightarrow{\sim} \phi^*E$, compatible under composition. 
        \item A morphism from $(E, \varphi)$ to $(E',\varphi')$ is a morphism $E\rightarrow E'$ that commutes with linearisations.
    \end{itemize}
\end{definition}

Note that $\Hom_{\calD_G}((E,\varphi), (E', \varphi'))= (\Hom_\calD(E,E'))^G.$

\begin{remark}
    Let $\calD$ be either $D^b(X)$ or $\Ku(X)$, with $G=\bZ/2\bZ=\langle \phi\rangle\subset \Aut(X)$.
    Then $E\in \calD$ being $G$-linearised is equivalent to $E$ being invariant under the pullback $\phi^*$. 
\end{remark}

We will need to know how semi-orthogonal decompositions behave under taking the equivariant category. 
\begin{theorem}\cite[Theorem 6.3]{Ela11}\label{thm: Ela11}
    Let $X$ be a quasi-projective variety, with an action of a finite group $G$. Let $D^b(X)=\langle \calA_1, \dots, \calA_n\rangle$ be a semi-orthogonal decomposition preserved by $G$. Then there is a semi-orthogonal decomposition of equivariant categories
    \[D^b_G(X)=\langle \calA_{1 G},\dots \calA_{n G}\rangle.\]
\end{theorem}

\begin{proposition}\cite[Proposition 3.3]{KP16}\label{prop: KP16}
    Let $\calA$ be a triangulated category with a trivial action of a finite group $G$. If $\calA_G$ is also triangulated, then there is a completely orthogonal decomposition:

    $$\calA_G= \langle \calA_G\otimes V_0,\dots \calA_G\otimes V_n\rangle,$$ where $V_0,\dots V_n$ are all irreducible representations of the finite group $G$.
\end{proposition}

Since irreducible representations $V_i$ are completely determined by their character $\chi_i$, we denote by $\calA_G\otimes \chi_i= \calA_G\otimes V_i$.

In particular, if $X\subset \bP^5$ is a cubic fourfold, then the derived category has the semi-orthogonal decomposition:
$$D^b(X)=\langle \Ku(X), \calO_X, \calO_X(1), \calO_X(2)\rangle.$$ If further $G=\bZ/2\bZ\subset \Aut(X),$ then $G$ preserves $\Ku(X), \calO_X, \calO_X(1)$ and $\calO_X(2).$ It follows that $\Ku(X)$ inherits an action by $G$. Applying the above results gives a semi-orthogonal decomposition of $D^b_G(X)$:

\begin{lemma}
    Let $X\subset \bP^5$ be a smooth cubic fourfold with $G=\bZ/2\bZ\subset \Aut(X).$ Then $G$ acts on $D^b(X)$ and $\Ku(X)$, and we have:
    $$D^b_G(X)=\langle \Ku_G(X), \calO_X, \calO_X\otimes \chi_1, \calO_X(1), \calO_X(1)\otimes \chi_1, \calO_X(2), \calO_X(2)\otimes \chi_1\rangle,$$ where $\chi_1$ is the non-trivial character of $G$.
\end{lemma}

\subsection{Root stacks}\label{subsec: root stack}

Let $Z$ be a smooth projective variety over an algebraically closed field $k$ of characteristic zero. Letting $\mathbb{G}_m$ act on $\mathbb{A}^1_k$, there is an equivalence between morphisms $Z\rightarrow [\mathbb{A}^1_k/\mathbb{G}_m]$ to the quotient stack $[\mathbb{A}^1_k/\mathbb{G}_m]$ and pairs $(L,s)$, where $L$ is an invertible sheaf on $Z$ and $s\in \Gamma(X,L)$ \cite[5.13]{Ols03}.  

For $D$ an effective Cartier divisor on $Z$ and $n$ a positive integer, the \textbf{$n$-th root stack} $\sqrt[n]{Z/D}$ is the fiber product
\[
\begin{tikzcd}
\sqrt[n]{Z/D} \arrow[r]\arrow[d] & {[\mathbb{A}^1_k/\mathbb{G}_m]}\arrow[d, "\theta_n"] \\
Z \arrow[r, "\delta"] & {[\mathbb{A}^1_k/\mathbb{G}_m]} , 
\end{tikzcd}
\]
where $\theta_n\colon [\mathbb{A}^1/\mathbb{G}_m]\rightarrow [\mathbb{A}^1_k/\mathbb{G}_m]$ is the morphism induced by taking $n$-th powers of both $\mathbb{A}^1_k$ and $\mathbb{G}_m$ and $\delta\colon Z\rightarrow [\mathbb{A}^1_k/\mathbb{G}_m]$ is the morphism associated to the pair $(\mathcal{O}(D), s_D)$ with $s_D$ the tautological section of $\mathcal{O}_Z(D)$ vanishing on $D$. See \cite{Cad07}, \cite{AGV08} for more on this construction. The $n$-th root stack $\sqrt[n]{Z/D}$ is a Deligne--Mumford stack \cite[Corollary 2.3.4]{Cad07}. Informally, we may view the root stack construction as modifying $Z$ along the divisor $D$ so as to get a stack with stablizer $\mu_n$ along $D$. 

In the setting of this paper, if $G=\mathbb{Z}/2\bZ$ acts on a variety $\widetilde{X}$ with fixed locus the divisor $D$ and $Z=\widetilde{X}/G$, then we obtain an identification of stacks:
\[\sqrt[2]{Z/D}=[\widetilde{X}/G],\]
where $[\widetilde{X}/G]$ is the quotient stack. 

We will describe $D^b([\widetilde{X}/G])$ by using the following description of $D^b\left(\sqrt[n]{Z/D} \right)$.

\begin{theorem}\label{thm: BD thm}\cite[Theorem 1.6]{IU15}, \cite[Theorem 1.2]{BD24} There is a semi-orthogonal decomposition
\[
D^b\left(\sqrt[n]{Z/D}\right)=\left\langle D^b(Z), D^b(D), D^b(D)\otimes \mathcal{O}(1), \ldots, D^b(D)\otimes \mathcal{O}(n-2)  \right \rangle. 
\]
    
\end{theorem}

\section{Hochschild homology of $\Ku_G(X)$}\label{sec:indecomp}

Throughout, we let $k$ be an algebraically closed field of characteristic zero and $X\subset \bP^5_k$ be a general cubic fourfold over $k$ admitting a symplectic involution $\phi\in \Aut(X).$ The geometry of such involutions were studied in \cite{Marq23}. We let $G:=\langle \phi\rangle \cong \bZ/2\bZ$. 

In this section we prove the following:
\begin{proposition}\label{prop: indecomp}
    The equivariant Kuznetsov component $\Ku_G(X)$ is a $K3$ category.
\end{proposition}
Since the action of $G$ on $X$ is symplectic, the equivariant Kuznetsov component $\Ku_G(X)$ is a 2-Calabi--Yau category (see \cite[Sections 6.3, 6.4]{beckmann2020}). Thus the main content of this section is to prove that the Hochschild (co)homology of $\Ku_G(X)$ coincides with that of a $K3$ surface.

By \Cref{thm: Ela11}, the equivariant Kuznetsov component $\Ku_G(X)$ appears as an admissible subcategory of the equivariant category $D^b_G(X)$. Classically, the Mckay correspondence \cite{BKR01} is used to understand $D^b_G(X)$ -- a crepant resolution of the quotient $X/G$ is constructed using the $G$-Hilbert scheme. However, this approach only works when the fixed locus is equidimensional, which as we explain in \Cref{subsec:geometry of X}, is not the case here. Instead, we will first blow up the fixed locus of $G$ and extend the action to $\widetilde{X}:=\Bl_{\Fix_G(X)}X$. In \Cref{subsec: Geometry of quotient}, we study the geometry of both $\widetilde{X}$ and $\widetilde{X}/G$. In \Cref{subsec: two decomps for equiv}, we obtain two different semi-orthogonal decompositions of $D^b_G(\widetilde{X})$. Finally, in \Cref{subsec: study HH}, we study the Hochschild homology of $D^b_G(\tX)$ via these decompositions in order to prove \Cref{prop: indecomp}.

\subsection{The geometry of $X$}\label{subsec:geometry of X}
Up to a linear change of coordinates, we can assume \[\phi\colon [x_0,x_1,x_2,x_3,x_4,x_5]\mapsto [x_0, x_1, x_2, x_3, -x_4, -x_5],\] 
and so $X\subset \bP^5$ has equation of the form
 \begin{align*}\label{eqn of X}
     g(x_0, x_1, x_2, x_3)+&x_4^2l_1(x_0, x_1,x_2,x_3)+2x_4x_5l_2(x_0,x_1,x_2,x_3)+x_5^2l_3(x_0,x_1,x_2,x_3)=0,
    \end{align*} 
where $l_1, l_2, l_3\in k[x_0,\dots, x_3]$ are linear forms and $g\in k[x_0,\dots, x_3]$ is a homogeneous cubic polynomial. We will also use that $X$ is of the equivalent form  
\begin{equation}\label{eqn: X in quadric form}
   g(x_0, x_1, x_2, x_3)+ x_0q_0(x_4,x_5)+x_1q_1(x_4,x_5)+x_2q_2(x_4,x_5)+x_3q_3(x_4,x_5)=0, 
\end{equation}
where $q_0, q_1, q_2, q_3 \in k[x_0,\dots, x_3]$ are homogeneous quadratic polynomials.

As an action on $\bP^5$, the involution $\phi$ has fixed locus given by the following line and projective $3$-plane:
\begin{align*}
L&:=V(x_0,x_1,x_2,x_3)\subset \bP^5,\\
\Pi&:=V(x_4,x_5)\subset \bP^5.
\end{align*}
The fixed locus of the action of $G$ on $X$ is then given by $\Fix_G(X)=L\cup \Sigma,$ where $\Sigma=\Pi\cap X.$ Note that, in equations, we have $\Sigma:=V(g(x_0,x_1,x_2,x_3))\subset \Pi$, so $\Sigma$ is a cubic surface.

The line $L\subset \Fix_G(X)$ gives a $G$-conic bundle structure on $X$ by projecting onto the disjoint $\bP^3$ given by $\Pi$, the geometry of which was studied in \cite[Section 4.2]{Marq23}. We let $\pi\colon \Bl_LX\rightarrow \Pi$ be the conic bundle with induced $G$-action. The discriminant locus is the union $\Sigma\cup Q$ where $Q:=V(l_1l_3-l_2^2)$ is a quadric cone. Let 
$$\tau\colon S\rightarrow \Sigma$$ be the component of the discriminant double cover (branched over $\Sigma\cap Q)$ associated to the conic bundle $\pi$. Then $S$ is a $K3$ surface, equipped with an anti-symplectic involution (the covering involution). We note the following:

\begin{lemma}\label{lem: K3}
    The $K3$ surface $S$ is naturally identified with a component of the fixed locus for the induced action of $G$ on the Fano variety of lines $F(X).$
\end{lemma}
\begin{proof}
    The surface $S$ parametrises lines that intersect both the fixed line $L$ and the cubic surface $\Sigma.$ Thus each line parametrised by $S$ in invariant under the involution.
\end{proof}

\subsection{The geometry of the quotient}\label{subsec: Geometry of quotient}
Let $\widetilde{X}:=\Bl_{Fix_G(X)}X$ and $f\colon \tX\rightarrow X$ be the blow up with exceptional divisors $E_L$ and $E_\Sigma$ over $L$ and $\Sigma$ respectively. We still denote by $G$ the extended action on $\tX$. 
We let $$q\colon \widetilde{X}\rightarrow \widetilde{X}/G$$ be the quotient map, and let $D_L:=q(E_L), D_\Sigma:=q(E_\Sigma).$ 
Note that $\widetilde{X}/G$ is smooth. Inspired by \cite[Proposition 3.6]{casalainamartin2024}, we prove the following:

\begin{proposition}\label{prop: Z}
    Let $S\subset \bP^3_{[X_0:X_1:X_2:X_3]}\times \bP^1_{[X_4:X_5]}$ be defined by
    \[S:=V\left( g(X_0,X_1,X_2,X_3), X_0q_0(X_4,X_5)+X_1q_1(X_4,X_5)+X_2q_2(X_4,X_5)+X_3q_3(X_4,X_5)\right),\]
    where $g, q_i$ come from the equation \eqref{eqn: X in quadric form} for $X$.
    Then $Z:=\Bl_S(\bP^3\times \bP^1)$ is isomorphic to $\tX/G.$
\end{proposition}
\begin{remark}
    Notice that the image of $S$ under the projection $\bP^3\times \bP^1\rightarrow \bP^3$ is exactly $\Sigma,$ and the projection has degree 2. Thus the $S$ in \Cref{prop: Z} is exactly the $K3$ double cover of $\Sigma$ as in \Cref{lem: K3}, explaining the abuse of notation.
\end{remark}
    
\begin{proof}
    The proof is very similar to that of \cite[Proposition 3.6]{casalainamartin2024}, but in one dimension higher. We shall verify the isomorphism $Z\cong \tX/G$ on an open affine cover. 
    
    We consider $\tX\subset \Bl_{L\sqcup \Pi}\bP^5$, where $L\sqcup \Pi$ is the fixed locus for the $G$-action on the ambient $\bP^5$.
    Notice that $\Bl_{L\sqcup \Pi}\bP^5$ is a subvariety of $\bP^5\times \bP^3_{[y_0:y_1:y_2:y_3]}\times \bP^1_{[y_4:y_5]}$ with the following equations:
\begin{align*} 
&\hspace{2mm} x_4y_5=x_5y_4;\\
x_0y_1=x_1y_0;&\hspace{2mm} x_0y_2=x_2y_0; \hspace{2mm} x_0y_3=x_3y_0;\\
x_1y_2=x_2y_1;& \hspace{2mm} x_1y_3=x_3y_1; \hspace{2mm} x_2y_3=x_3y_2.
\end{align*}
By taking the strict transform of $X$ and then quotienting by $G$, we obtain a local description of $\tX/G$. For instance, let $x_0=y_0=y_4=1$: then $\tX/G$ has the following local equation in $\bA^5$ with coordinates $(y_1, y_2, y_3, y_5, a_4)$ where $a_4=x_4^2$:
\begin{equation*}
    g(1, y_1, y_2,y_3)+a_4\left(q_0(1,y_5)+y_1q_1(1,y_5)+y_2q_2(1,y_5)+y_3q_3(1,y_5)\right)=0.
\end{equation*}

On the other hand, $Z:=\Bl_S(\bP^3\times \bP^1)$ is given by the equation:
\[Y_0(X_0q_0(X_4,X_5)+X_1q_1(X_4,X_5)+X_2q_2(X_4,X_5)+X_3q_3(X_4,X_5)) + Y_1 g(X_0,X_1,X_2,X_3)=0\]
in $\bP^3\times \bP^1\times \bP^1$ with coordinates 
    $[X_0:X_1:X_2:X_3]\times[X_4:X_5]\times [Y_0: Y_1].$

Taking $X_0=X_4=Y_1=1,$ the affine variety is given by:
\[V(Y_0(q_0(1,X_5)+X_1q_1(1,X_5)+X_2q_2(1,X_5)+X_3q_3(1,X_5)) +  g(1,X_1,X_2,X_3)) \subset \bA^5.\] 
By mapping:
\begin{align*}
    a_4\mapsto Y_0\\
    y_5\mapsto X_5\\
    y_i\mapsto X_i
\end{align*} for $i=1,2,3$, we get an isomorphism between the corresponding affine subvarieties. Similarly, one can check that there are isomorphisms between other pairs of affine opens and that these local isomorphisms glue together.
\end{proof}
We will use this description of $Z:=\tX/G$ to give a semi-orthogonal decomposition of $D^b_G(\tX)$. Applying Orlov's blow up formula \cite[Proposition 11.18]{HuyFMBook}, we immediately see:

\begin{lemma}\label{semi-orth 1}
We have the following semi-orthogonal decomposition for $D^b(Z)$:
    \begin{align*}
    D^b(Z)=\langle D^b(S), \calO_Z, \calO_Z(1,0), \calO_Z(2,0), \calO_Z(3,0), \calO_Z(0,1), \calO_Z(1,1), \calO_Z(2,1), \calO_Z(3,1)\rangle,
\end{align*}
where $\calO_Z(i,j)$ is the pullback of $\calO(i,j)$ under the blow up $Z\rightarrow \bP^3\times \bP^1$.
\end{lemma}

\subsection{Two decompositions for $D^b_G(\widetilde{X})$}\label{subsec: two decomps for equiv} 
We are now equipped to study the equivariant category $D^b_G(\tX)$. We immediately obtain the following semi-orthogonal decomposition:

\begin{lemma}\label{lem: first decomp}
    We have 
    \begin{align*}
        D^b_G(\tX)=\langle  \Ku_G(X), \calO_{\tX}, \calO_{\tX}\otimes \chi_1, \calO_{\tX}(1), \calO_{\tX}(1)\otimes \chi_1,\calO_{\tX}(2),\\
        \calO_{\tX}(2)\otimes \chi_1, D^b(\bP^1), D^b(\bP^1)\otimes \chi_1, D^b(\bP^1)\otimes\calO_{E_L}(1), \\D^b(\bP^1)\otimes\calO_{E_L}(1)\otimes \chi_1, D^b(\Sigma), D^b(\Sigma)\otimes \chi_1\rangle,
    \end{align*}
    where $\chi_1$ is the non-trivial character of $G$.
\end{lemma}
\begin{proof}
     We first use Orlov's blow up formula to see that $D^b(\tX)=\langle D^b(X), D^b(L), D^b(L)\otimes \calO_{E_L}(1),D^b(\Sigma)\rangle$. Identifying $L\cong \bP^1$ and taking the $G$-equivariant category as in \Cref{thm: Ela11} gives the result.
\end{proof}

We wish to compare the decomposition above with another semi-orthogonal decomposition of $D^b_G(\tX)$. Recall that $D^b_G(\tX):=D^b([\tX/G]),$ where $[\tX/G]$ is the quotient stack. We obtain a second decomposition by considering $[\tX/G]$ as a root stack, as in \Cref{subsec: root stack}. 

\begin{proposition}\label{prop: second decomp}
    We have:
    \begin{align*}
        D^b_G(\tX)=\langle D^b(Z), D^b(D_L), D^b(D_\Sigma)\rangle. 
    \end{align*}
\end{proposition}
\begin{proof}This follows immediately from the semi-orthogonal decomposition of Theorem \ref{thm: BD thm} together with the identification $\tilde{X}/G\cong Z$ of Proposition \ref{prop: Z} and the fact that $D_L$ and $D_\Sigma$ are disjoint. 
\end{proof}

\subsection{Hochschild homology of $\Ku_G(X)$}\label{subsec: study HH}
Finally, we are able to prove \Cref{prop: indecomp}, that the equivariant category $\Ku_G(X)$ is a $K3$ category.

\begin{proof}[Proof of \Cref{prop: indecomp}]
    Since $\Ku_G(X)$ is a 2-Calabi--Yau category, it suffices to prove that $\Ku_G(X)$ has the same Hochschild homology as $D^b(S)$, where $S$ is a $K3$  surface. 
    We will use the two semi-orthogonal decompositions of $D^b_G(\tX)$ from Section~\ref{subsec: two decomps for equiv} to compute $\HH_{\bullet}(\Ku_G(X)).$ 

    First, using \Cref{prop: second decomp} and \Cref{semi-orth 1} we see that
    $$\HH_{i}(D^b_G(\tX))=\begin{cases}
        \HH_{-2}(D^b(S))=k, & i=-2\\
        \HH_{0}(D^b(S))\oplus k^{32}, & i=0\\
        \HH_{2}(D^b(S))=k, & i=2.
    \end{cases}$$
    Indeed, this follows since Hochschild homology is additive (\Cref{prop: HH additive}) and since both $D_L$ and $D_\Sigma$ are projective bundles over $\bP^1, \Sigma$, respectively. 
    Both $D^b(\bP^1)$, $D^b(\Sigma)$ have full exceptional collections, thus it follows that $D^b(D_L), D^b(D_{\Sigma})$ also admit full exceptional collections -- one can compute them explicitly using \cite[Corollary 8.36]{HuyFMBook}. By \Cref{lem: HH of except}, exceptional objects contribute homology in only the zeroth degree, and one can compute the total number of exceptional objects in the decomposition in \Cref{prop: second decomp}.

    Secondly, we use \Cref{lem: first decomp}, along with the observation that the computation above implies the vanishing of all odd-degree Hochschild homology of $D^b_G(\tX)$, to see that $$\HH_{i}(D^b_G(\tX))=\begin{cases}\HH_{-2}(\Ku_G(X)), & i=-2,\\
    \HH_0(\Ku_G(X)) \oplus k^{32}, & i=0\\
    \HH_{2}(\Ku_G(X)), & i=2.
    \end{cases}$$
    Again, this follows since both $D^b(\bP^1), D^b(\Sigma)$ have full exceptional collections. As above, one checks that the full semi-orthogonal decomposition of \Cref{lem: first decomp} has 32 exceptional objects.
    
    Comparing homology then yields that $\HH_\bullet(\Ku_G(X))=\HH_\bullet(D^b(S)).$
    \end{proof}

\section{The equivalence}\label{sec:equivalence}

In this section, we prove the the equivalence of \Cref{main theorem}:
$$\Ku_G(X)\simeq D^b(S),$$
where $S\subset F(X)$ is the $K3$ component of the fixed locus of $G$ acting on $F(X)$.

First, we define a functor $\Phi\colon D^b(S)\rightarrow D^b(X),$ by specifying a Fourier-Mukai kernel in $D^b(S\times X).$
Recall from \Cref{subsec:geometry of X} that a point $s\in S$ determines a line $\ell_s\subset X$ that is invariant under the involution $\phi$. 
Hence each $s\in S$ determines a sheaf $I_{\ell_s}$, the ideal sheaf of the line $\ell_s$. The sheaf $I_{\ell_s}(1)$ sits in the exact sequence:
\begin{equation}\label{SES}
0\rightarrow F_{\ell_s}\rightarrow \calO_X^{\oplus 4}\rightarrow I_{\ell_{s}}(1)\rightarrow 0,
\end{equation}
where $\calO_X^{\oplus 4}=H^0(X, I_{\ell_s}(1))\otimes \calO_X\rightarrow I_{\ell_s}(1)$ is the evaluation map, and $F_{\ell_s}$ is the kernel. 

\begin{lemma}\label{lem:prop of F_s}
    The sheaf $F_{\ell_s}$ is a reflexive sheaf of rank 3 on $X$, invariant under the action of $G$ (i.e. $F_{\ell_s}\cong \phi^*F_{\ell_s})$, and $F_{\ell_s}\in \Ku(X).$ Further, for $s\neq t$ we have $F_{\ell_s}\not\cong F_{\ell_t}.$
\end{lemma}

\begin{proof}
    The line $\ell_s\subset X$ is invariant under the action of $G$, and hence $I_{\ell_s}\cong \phi^*I_{\ell_s}$. Since $\phi$ is a finite map, the pullback $\phi^*$ is exact and it follows that $F_{\ell_s}\cong \phi^*F_{\ell_s}$. The remaining claims follows from combining \cite[Lemma 5.1 and Proposition 5.2]{KM09}.
\end{proof}

Let $\calI$ be the restriction of the ideal sheaf of the universal line on $F(X)\times X$ to $S\times X$, and let $\calF$ be the universal sheaf on $S\times X$ given as the kernel of $\calO_{S\times X}^{\oplus 4}\twoheadrightarrow \calI(1).$ 
Thus $\calF|_{\{s\}\times X}\cong F_{\ell_s}$ for all $s\in S.$
We define the functor $\Phi\colon D^b(S)\rightarrow D^b(X)$ to be the Fourier-Mukai transform with kernel $\calF\in D^b(S\times X).$

\begin{proposition}
    The functor $\Phi$ takes values in $\Ku(X).$ Moreover, if we equip $D^b(S)$ with a trivial $G$-action, the functor $\Phi$ is a $G$-functor.
\end{proposition}
\begin{proof}
    Let $\Phi^L$ be the left adjoint to $\Phi$, which exists by \cite[Proposition 5.9]{HuyFMBook}.
    Note that $\Phi(\calO_s)=F_{\ell_s}\in \Ku(X)$, by \Cref{lem:prop of F_s}. 
    Hence we see that
    $$\Ext^i(\Phi^L(\calO_X(j)), \calO_s)\cong \Ext^i(\calO_X(j), F_{\ell_s}) = 0$$ for $-1\leq j \leq 1$ and all $i$. 
    This means $\Phi^L(\calO_X(j))=0$ for $-1\leq j\leq 1$, and so the image of $\Phi$ lies in $\Ku(X).$

    To show that $\Phi$ is a $G$-functor, since $G=\bZ/2\bZ$ it is enough to show that there is a natural isomorphism $\Phi \xrightarrow{\sim} \phi^*\Phi$, which will follow from showing that the kernel $\calF$ is invariant under the diagonal action of $G$ on $S\times X$. 
    We claim that the natural morphism $\phi^*\calF \to \calF$ given by the pullback is an isomorphism. 
    Considering that $\calF|_{\{s\}\times X} \cong F_{\ell_s}$ for all $s\in S$, we have the following commutative diagram for any $(s,x)\in S\times X$:
    \[
    \xymatrix{
    \calF_{(s,\phi(x))}=(\phi^*\calF)_{(s,x)} \ar[r] \ar[d]_[@!-90]{\sim} & \calF_{(s,x)} \ar[d]^[@!-90]{\sim} \\
    (F_{\ell_s})_{\phi(x)}=(\phi^*F_{\ell_s})_x \ar[r]^-{\sim} & (F_{\ell_s})_x
    }
    \]
    where the bottom arrow in the diagram is also an isomorphism by \Cref{lem:prop of F_s}.
\end{proof}

By \cite[Prop.~3.1]{beckmann2020} and the fact that $\Phi$ is a $G$-functor, there exists a unique functor $\Phi_G\colon D^b(S) \to \Ku_G(X)$ through which $\Phi$ factors:
\[
\xymatrix{
 & \Ku_G(X) \ar[d] \\
D^b(S) \ar[r]_{\Phi} \ar@{-->}[ru]^{\Phi_G} & \Ku(X), 
}
\]
where the map $\Ku_G(X)\rightarrow \Ku(X)$ forgets the $G$-linearisation.
We will show that $\Phi_G$ is an equivalence of categories, thereby proving \Cref{main theorem}. 

We wish to use the criteria of Bondal and Orlov \cite[Proposition 7.1]{HuyFMBook} to prove that $\Phi_G$ is fully faithful. In order to do so, we need to analyze some extension groups. Note that for $s,t\in S$, we have $\Ext^i_{\Ku_G(X)}(F_{\ell_s}, F_{\ell_t})\cong(\Ext_{\Ku(X)}^i(F_{\ell_s}, F_{\ell_t}))^G,$ where $G$ acts on $\Ext_{\Ku(X)}^i(F_{\ell_s}, F_{\ell_t})$ as in \cite[Section 3.1]{beckmann2020}. 

\begin{lemma}\label{lem:extcalc}
For all $s\in S$, we have
\[\dim \Ext^i_{\Ku_G(X)}(F_{\ell_s}, F_{\ell_s})= \begin{cases}
    1 & \text{if } i=0, 2\\
    2 & \text{if } i=1,\\
    0 & \text{ otherwise.}
\end{cases}\]
\end{lemma}
\begin{proof}
First, let $\ell\subset X$ be any line (not necessarily parametrised by $S$). The following dimensions were computed in \cite[Section 5]{KM09}:
\[\dim \Ext^i_{\Ku(X)}(F_{\ell}, F_{\ell})= \begin{cases}
    1 & \text{if } i=0, 2\\
    4 & \text{if } i=1,\\
    0 & \text{ otherwise.}
\end{cases}\]

Now we restrict to lines $\ell_s\subset X$ for $s\in S$. By \Cref{lem:prop of F_s}, the sheaf $F_{\ell_s}\cong \phi^*F_{\ell_s}$ and is naturally $G$-linearised, defining an object in $\Ku_G(X).$ Note that any $f\in \Hom(F_{\ell_s}, F_{\ell_s})$ is automatically $G$-invariant, hence $\dim\Hom_{\Ku_G(X)}(F_{\ell_s},F_{\ell_s})=1.$ Since $\Ku_G(X)$ is a 2-Calabi--Yau category, we also have $\dim\Ext^2_{\Ku_G(X)}(F_{\ell_s},F_{\ell_s})=1.$

It remains to compute $\Ext^1_{\Ku_G(X)}(F_{\ell_s},F_{\ell_s})= (\Ext^1_{\Ku(X)}(F_{\ell_s},F_{\ell_s}))^G.$ Let $F'(X)$ be the moduli space of stable sheaves on $X$ containing the sheaves $F_{\ell}$ for any $\ell\subset X$. In \cite[Proposition 5.5]{KM09}, the authors prove that the map $F(X)\rightarrow F'(X)$ sending $\ell\mapsto F_\ell$ is an isomorphism of $F(X)$ on a connected component of $F'(X)$. We identify $F(X)$ with its image. By \cite[Corollary~4.5.2]{HuybrechtsLehn}, $\Ext^1(F_\ell, F_\ell)$ is isomorphic the tangent space to $F'(X)$ at the point $F_\ell$. The group $G$ acts on $F(X)$, with a 2-dimensional smooth fixed locus $S\subset F(X)$, and the image of $S$ is exactly the sheaves $F_{\ell_s}$. Since the tangent space of the fixed locus at a point is isomorphic to the $G$-invariant tangent space of $F(X)$, i.e. $T_{F_{\ell_s}}S = (T_{F_{\ell_s}}F(X))^G$, we see that 
\[2= \dim S = (\Ext^1_{\Ku_G(X)}(F_{\ell_s},F_{\ell_s}))^G = \Ext^1_{\Ku_G(X)}(F_{\ell_s},F_{\ell_s}),\]
as desired.
\end{proof}

\begin{proposition}\label{prop:fullyfaithful}
    The functor $\Phi_G \colon D^b(S)\rightarrow \Ku_G(X)$ is fully faithful.    
\end{proposition}

\begin{proof}
    We will use the criteria of Bondal and Orlov \cite[Proposition~7.1]{HuyFMBook}. In particular, we need to show:
        $$\dim \Ext_{\Ku_G(X)}^i(\Phi_G(\calO_s), \Phi_G(\calO_t))=\begin{cases}
    1 & \text{if } s=t, i=0\\
    0 & \text{if } s\neq t, \text{ or } i<0 \text{ or } i>2.
\end{cases}$$
    This is equivalent to showing:
    $$\dim \Ext_{\Ku_G(X)}^i(F_{\ell_s}, F_{\ell_t})=\begin{cases}
    1 & \text{if } s=t, i=0\\
    0 & \text{if } s\neq t, \text{ or } i<0 \text{ or } i>2.
\end{cases}$$

The fact that $\dim \Ext_{\Ku_G(X)}^0(F_{\ell_s}, F_{\ell_s})=1$ follows from \Cref{lem:extcalc}, and the fact that $\dim \Ext_{\Ku_G(X)}^i(F_{\ell_s}, F_{\ell_t})=0$ for $i<0$ or $i>2$ follows because $\Ku_G(X)$ is a $2$-Calabi--Yau category.

Now let $s, t\in S$ with $s\neq t$. Since $F_{\ell_s}\in \Ku(X),$ we have that $\Hom(\calO_X, F_{\ell_s})=0$ for all $s\in S.$ We consider the long exact sequence obtained by applying $\Hom(-, F_{\ell_t})$ to the sequence \eqref{SES}: since $\Hom(\calO_X, F_{\ell_t})^{\oplus 4}$ surjects onto $\Hom(F_{\ell_s}, F_{\ell_t})$, we see that the latter is also zero. Applying Serre duality gives also $\Ext^2(F_{\ell_s}, F_{\ell_t})=0.$
Thus taking $G$-invariants gives the necessary vanishing of $\Ext^i_{\Ku_G(X)}(F_{\ell_s}, F_{\ell_t})$ for $i=0,2$.

It remains to show that $\Ext^1_{\Ku_G(X)}(F_{\ell_s}, F_{\ell_t})=0.$ For this, we compute the Euler characteristic $\chi_{\Ku_G(X)}(F_{\ell_s},F_{\ell_t})$, which is deformation invariant, hence equal to $\chi_{\Ku_G(X)}(F_{\ell_s},F_{\ell_s})$. By \Cref{lem:extcalc}, $\chi_{\Ku_G(X)}(F_{\ell_s},F_{\ell_s})=1-2+1=0$, which implies the desired vanishing of $\Ext^1_{\Ku_G(X)}(F_{\ell_s}, F_{\ell_t})$.
\end{proof}

\begin{proof}[Proof of \Cref{main theorem}]
The functor $\Phi_G$ is an equivalence of categories. This follows immediately by Propositions \ref{prop:fullyfaithful} and \ref{prop: indecomp}, noting that a fully faithful Fourier-Mukai functor between two $K3$ categories is automatically an equivalence (since the image is an admissible subcategory and the target is indecomposable, cf.~\Cref{lem: Bridgelandtrick}).
\end{proof}

\bibliographystyle{alpha}
\bibliography{bibliography}
\end{document}